\newif\ifsvjour
\ifsvjour
	\documentclass[smallextended]{svjour3}
	\journalname{Graphs and Combinatorics}
	\usepackage{amsmath,amssymb,mathptmx}
\else
	\documentclass{amsart}
	\newtheorem{theorem}{Theorem}[section]

	\newtheorem{lemma}[theorem]{Lemma}
	
	\newtheorem{proposition}[theorem]{Proposition}

\fi

\newcommand{\cgF}{\mathcal{F}}

\newcommand{\cgR}{\mathcal{R}}

\newcommand{\cgI}{\mathcal{I}}

\newcommand{\Min}{\operatorname{Min}}
\newcommand{\Max}{\operatorname{Max}}

\newcommand{\Inc}{\operatorname{Inc}}
\newcommand{\width}{\operatorname{width}}
\newcommand{\height}{\operatorname{height}}
\newcommand{\Idim}{\operatorname{Idim}}
\newcommand{\fdim}{\operatorname{dim}^*}
\newcommand{\fd}{\operatorname{fdim}}

\begin{document}

\ifsvjour

\title{Forcing Posets with Large Dimension to Contain Large Standard Examples}

\author{Csaba Bir\'o \and Peter Hamburger \and Attila P\'or \and William T. Trotter}

\institute{C. Bir\'o \at Department of Mathematics, University of Louisville, Louisville, KY 40292, U.S.A.\\ \email{csaba.biro@louisville.edu}
	\and
	P. Hamburger \at Department of Mathematics, Western Kentucky University, Bowling Green, KY 42101, U.S.A.\\ \email{peter.hamburger@wku.edu}
	\and
	A. P\'or \at Department of Mathematics, Western Kentucky University, Bowling Green, KY 42101, U.S.A.\\ \email{attila.por@wku.edu}
	\and
	W. T. Trotter \at School of Mathematics, Georgia Institute of Technology, Atlanta, GA 30332, U.S.A.\\ \email{trotter@math.gatech.edu}}

\subclass{06A07\and 05C35}

\keywords{poset\and dimension\and width\and standard example}

\else

\title[LARGE STANDARD EXAMPLES]{Forcing Posets with Large Dimension\\
  to Contain Large Standard Examples}
  
\author[C.~BIR\'{O}]{Csaba Bir\'{o}}

\address{Department of Mathematics\\
 University of Louisville\\
 Louisville, Kentucky 40292\\
 U.S.A.}

\email{csaba.biro@louisville.edu}

\author[P.~HAMBURGER]{Peter Hamburger}

\address{Department of Mathematics\\
 Western Kentucky University\\
 Bowling Green, Kentucky 42101\\
 U.S.A.}

\email{peter.hamburger@wku.edu}

\author[A.~P\'{O}R]{Attila P\'{o}r}

\address{Department of Mathematics\\
 Western Kentucky University\\
 Bowling Green, Kentucky 42101\\
 U.S.A.}

\email{attila.por@wku.edu}

\author[W.~T.~TROTTER]{William T. Trotter}

\address{School of Mathematics\\
  Georgia Institute of Technology\\
  Atlanta, Georgia 30332\\
  U.S.A.}

\email{trotter@math.gatech.edu}

\subjclass[2010]{06A07, 05C35}

\keywords{Poset, dimension, width, standard example}

\fi

\begin{abstract}
The dimension of a poset $P$, denoted 
$\dim(P)$, is the least positive integer $d$ for which $P$ is the 
intersection of $d$ linear extensions of $P$.  The maximum dimension 
of a poset $P$ with $|P|\le 2n+1$ is $n$, provided $n\ge2$, and this inequality
is tight when $P$ contains the standard example $S_n$.
However, there are posets with large dimension that do not 
contain the standard example $S_2$.  Moreover, for each fixed 
$d\ge2$, if $P$ is a poset with $|P|\le 2n+1$ and $P$ does not 
contain the standard example $S_d$, then 
$\dim(P)=o(n)$.  Also, for large $n$, there is a poset $P$ with
$|P|=2n$ and $\dim(P)\ge (1-o(1))n$ such that the largest $d$ so that
$P$ contains the standard example $S_d$ is $o(n)$. 
In this paper, we will show that 
for every integer $c\ge1$, there is an integer $f(c)=O(c^2)$ so that 
for large enough $n$, if $P$ is a poset with $|P|\le 2n+1$ and $\dim(P)\ge 
n-c$, then $P$ contains a standard example $S_d$ with $d\ge n-f(c)$.  
From below, we show that $f(c)=\Omega(c^{4/3})$.  On the other hand, 
we also prove an analogous result for fractional dimension, and in 
this setting $f(c)$ is linear in $c$.  Here the result is best possible 
up to the value of the multiplicative constant.
\end{abstract}

\maketitle

\section{Introduction}\label{sec:intro}

When $G$ is a graph, let $\chi(G)$ denote the \textit{chromatic
number} of $G$, and let $\omega(G)$ denote the \textit{maximum clique
size} of $G$.  As is well known, there are triangle-free graphs
(graphs with $\omega(G)\le 2$) with large chromatic number.
Moreover, just by analyzing the behavior of the Ramsey number 
$R(3,k)=\Theta(k^2/\log k)$, it follows that there
are triangle-free graphs on $n$ vertices with chromatic number 
$\Omega(\sqrt{n/\log n})$.  However, when a graph on $n$ vertices
has chromatic number close to $n$, it must have a large clique.
We state for emphasis the following self-evident 
proposition\footnote{Bir\'{o}, F\"{u}redi and Jahanbekam~\cite{bib:BiFuJa}
have studied the question of forcing large cliques in graphs in far greater
detail than this elementary proposition.  But this simple result suffices
in establishing a parallel line of thought in graph theory.}.

\begin{proposition}\label{pro:graph}
Let $c$ be a positive integer.  If $n>2c$ and $G$ is a graph on
$n$ vertices with $\chi(G)\ge n-c$, then $\omega(G)\ge n-2c$.
\end{proposition}

This paper is concerned with analogous results for finite partially
ordered sets (posets).

\subsection{Posets and Dimension}

We assume familiarity with basic notation and terminology for 
posets, including chains and antichains; comparable
and incomparable elements; minimal and maximal elements; and linear
extensions.  For readers who seek additional
background material on posets, Trotter's book~\cite{bib:Trot-Book}
is a good reference.

We denote by $|P|$ the number of elements of $P$ and
we frequently refer to elements of $P$ as points. Recall that the 
\textit{height} of a poset $P$, denoted $\height(P)$, is the 
maximum size of a chain in $P$, while the \textit{width} of $P$, 
denoted $\width(P)$, is the maximum size of an antichain in $P$. As 
there is no completely standard notation for this relation, here 
we will write $x\parallel y$ in $P$ when $x$ and $y$ are distinct and 
incomparable points in a poset $P$. 

A family $\cgR=\{L_1,L_2,\dots,L_d\}$ of linear
extensions of a poset $P$ is called a \textit{realizer} of
$P$ if $x<y$ in $P$ if and only if $x<y$ in $L_i$ for each $i=1,2,\dots,d$.
Equivalently, a family $\cgR$ of
linear extensions of $P$ is a realizer of $P$
if and only if for every ordered pair $(x,y)$ with $x\parallel y$ in $P$,
there is some $i$ with $1\le i\le d$ with $x>y$ in $L_i$.
Dushnik and Miller~\cite{bib:DusMil} defined
the \textit{dimension} of a poset $P$, denoted $\dim(P)$, as the 
least positive integer $d$ for
which $P$ has a realizer of size $d$.  When $P$ is a poset,
the \textit{dual} of $P$ is the poset $Q$ with the same
ground set as $P$ with $x>y$ in $Q$ if and
only if $x<y$ in $P$.  The following basic properties of dimension 
are self-evident.

\begin{enumerate}
\item  $\dim(P)=1$ if and only if $P$ is a chain.
\item If $Q$ is a subposet of $P$, then $\dim(Q)\le\dim(P)$.
\item If $Q$ is the dual of $P$, then $\dim(P)=\dim(Q)$
\item If $P$ is an antichain of size at least~$2$, then $\dim(P)=2$.
\end{enumerate}

The following construction was first noted by Dushnik and Miller 
in~\cite{bib:DusMil}.
For an integer $d\ge2$, let $S_d$ be the following height~$2$ poset:\quad
$S_d$ has $d$ minimal elements $\{a_1,a_2,\dots,a_d\}$ and $d$ maximal elements
$\{b_1,b_2,\dots,b_d\}$.  The partial ordering on $S_d$ is defined
by setting $a_i<b_j$ in $S_d$ if and only if $i\neq j$.  Evidently,
$\dim(S_d)\ge d$, since if $L$ is any linear extension of $S_d$, there
can be at most one value of $i$ with $a_i>b_i$ in $L$.  On the other hand,
it is easy to see that $\dim(S_d)\le d$.  The poset $S_d$ is called
the \textit{standard example} (of dimension~$d$).

Hiraguchi~\cite{bib:Hira} proved that if $n\ge2$ and $|P|\le 2n+1$,
then $\dim(P)\le n$. The family of standard examples shows that
this inequality is tight.  Moreover\footnote{The inductive step in
the proof of Theorem~\ref{thm:Kimb},
as presented by Kimble in~\cite{bib:Kimb}, is relatively compact, 
and some might even say that it is elegant.  On the other hand, no
entirely complete proof of the base case ($n=4$) has
ever been written down---nor is this likely to happen.  The problem is to
show that if $|P|=9$ and $\dim(P)=4$, then $P$ contains $S_4$.  The 
issue is that the analogous statement is not true when $n=3$, as there 
are $20$ posets of size~$7$ which have dimension~$3$ and do not
contain a $3$-dimensional subposet on $6$ points.}, we have the following 
theorem, with the even case due to Bogart and Trotter~\cite{bib:BogTro} 
and the odd case due to Kimble~\cite{bib:Kimb}.

\begin{theorem}\label{thm:Kimb}
If $n\ge 4$ and $|P|\le 2n+1$, then $\dim(P)<n$ unless $P$ contains the
standard example $S_n$.
\end{theorem}

If a poset $P$ contains a large standard example, then the dimension
of $P$ is large, and the following result, which is both
a generalization of Theorem~\ref{thm:Kimb} and a
poset analogue of Proposition~\ref{pro:graph}, is the first of the two 
principal results in this paper.

\begin{theorem}\label{thm:main-dim}
For every positive integer $c$, there is an integer $f(c)=O(c^2)$ so that if 
$n>10f(c)$ and $P$ is a poset with $|P|\le 2n+1$ and $\dim(P)\ge n-c$, 
then $P$ contains a standard example $S_d$ with $d\ge n-f(c)$.
\end{theorem}

In our proof, the function $f(c)$ will be quadratic in $c$, but
this may not be best possible.  However, we are able to show that
$f(c)=\Omega(c^{4/3})$.  Of course, the restriction $n>10f(c)$ in the statement of
Theorem~\ref{thm:main-dim} is just intended to make $n$ sufficiently
large in terms of $c$.  This restriction also serves to keep us safely
away from annoying small cases.

In the remainder of this introductory section, we briefly discuss
results which are more substantive than Proposition~\ref{pro:graph} and 
serve to motivate our main theorem.
In Section~\ref{sec:preliminaries}, we gather some essential preliminary
material, and the proof of our main theorem is given in Section~\ref{sec:proof}.

In Section~\ref{sec:fracdim}, we prove an analogous theorem for fractional
dimension, and in this setting, the function $f(c)$ is linear in $c$.
Of course, this result is best possible up to the value of the multiplicative
constant.  We close with some open problems in Section~\ref{sec:close}.

\subsection{Large Dimension without Large Standard Examples}

Dushnik and Miller~\cite{bib:DusMil}
made the following observation:\quad  For an integer $n\ge3$, let
$P(1,2;n)$ denote the poset consisting of the $1$-element and $2$-element
subsets of $\{1,2,\dots,n\}$, and let $d(1,2;n)=\dim(P(1,2;n))$.  
Using the classic theorem of Erd\H{o}s
and Szekeres, they noted that $\dim(1,2;n)=\Omega(\lg\lg n)$.
While $P(1,2;n)$ contains $S_3$, it does not contain $S_d$ for any $d\ge 4$.

We pause to comment that much more can be said about the growth
rate of $d(1,2;n)$. By combining results of Ho\c{s}ten and
Morris~\cite{bib:HosMor} with estimates of Kleitman and 
Markovsky~\cite{bib:KleMar}, the following theorem follows in a
straightforward manner.

\begin{theorem}\label{thm:12n}
For every $\epsilon>0$, there is
an integer $n_0$ so that if $n>n_0$ and 
\[
s= \lg\lg n + 1/2 \lg\lg\lg n +1/2\lg{\pi} + 1/2,
\]
then $s-\epsilon < d(1,2;n)< s +1 +\epsilon$.
\end{theorem}

As a consequence, for almost all large values of $n$, we can
compute the value of $d(1,2;n)$ \textit{exactly}; for the
remaining small fraction of values, we are able to compute two
consecutive integers and say that $d(1,2;n)$ is one of the
two.

But a poset can have large dimension without containing $S_3$.
In \cite{bib:FelTro}, Felsner and Trotter observed that if $d$ and $g$
are positive integers, then there is a poset $P$ of height~$2$ for
which the girth of the comparability graph of $P$ is at least $g$
while the dimension of $P$ is at least $d$.  In fact, this observation is an
immediate consequence of the well known fact that there is a graph
$G$ whose girth is at least $g$ and whose chromatic number is at least~$d$
(see~\cite{bib:FelTro} and~\cite{bib:FeLiTr} for additional details
on adjacency posets and related problems).  Posets with large dimension
and large girth will contain the standard example $S_2$.  However, they
do \textit{not} contain the standard example $S_3$.

To close this story, we note that a poset can have large dimension 
without containing $S_2$.
A poset $P$ is called an \textit{interval order} if there is
a family $\cgI=\{[a_x,b_x]: x\in P\}$ of closed intervals of the
real line $\mathbb{R}$ so that $x<y$ in $P$ if and only if $b_x<a_y$
in $\mathbb{R}$. Fishburn~\cite{bib:Fish} showed that a poset $P$ is an
interval order if and only if it does not contain the standard
example $S_2$.  We note that $S_2$ is isomorphic to $\mathbf{2}+
\mathbf{2}$, the disjoint sum of two $2$-element chains. 
In general posets of height~$2$ can have arbitrarily large dimension.
However, Rabinovitch~\cite{bib:Rabi} showed that the dimension of
an interval order is bounded in terms of its height. Specifically, he
showed that the maximum dimension $d_n$ of an interval order of height
$n$ is $O(\log n)$.

In~\cite{bib:BoRaTr}, Bogart, Rabinovitch and Trotter considered
the \textit{canonical} interval order $I(n)$ consisting of all intervals with
distinct integer end points in $\{1,2,\dots,n\}$ and showed that 
$\dim(I(n))$ goes to infinity with $n$.  
However, by combining the results of F\"{u}redi, Hajnal, R\"{o}dl and 
Trotter~\cite{bib:FHRT} with the same analysis used for Theorem~\ref{thm:12n},
it is easy to verify that $|\dim(I(n))-d(1,2:n)|\le 5$ and
$|d_n-d(1,2;n)|\le 5$, for all $n\ge2$.  So the values of
$\dim(I(n))$ and $d_n$ can be computed ``almost exactly.''

\subsection{Forcing Large Standard Examples}

With this background in mind, it is natural to ask whether there
are conditions which force a poset of large dimension to contain
a large standard example.  A partial answer is provided 
by the following theorem of Bir\'{o}, Hamburger and P\'{o}r~\cite{bib:BiHaPo-1}.

\begin{theorem}\label{thm:BHP}
For every integer $d\ge2$ and every $\epsilon >0$, there is an
integer $n_0$ so that if $n>n_0$ and $P$ is a poset with $|P|\le 2n+1$ 
and $P$ does not contain the standard example $S_d$, then $\dim(P)<\epsilon n$.
\end{theorem}

Paralleling our earlier discussion on chromatic number,
it was natural for Bir\'{o}, F\"{u}redi and Jahanbekam to 
conjecture\footnote{To be precise, in~\cite{bib:BiFuJa},
Bir\'{o}, F\"{u}redi and Jahanbekam conjectured that for a fixed small 
$\epsilon>0$, a poset $P$ with at most $2n+1$ points and dimension at least 
$(1-\epsilon)n$ must contain a standard example $S_d$ with $d$ a positive 
fraction of $n$.  Examples~1 and 2 show that this conjecture is too strong.
Nevertheless, our Theorem~\ref{thm:main-dim} confirms the basic
intuition behind their conjecture.} in~\cite{bib:BiFuJa} that 
when $n$ is large and $P$ is
a poset with $|P|\le 2n+1$, if the dimension of $P$ is close to $n$,
then $P$ must contain a standard example $S_d$ with $d$ also
close to $n$. 

Here are two examples to show that when $|P|\le 2n+1$, the dimension 
must be quite close to $n$ in order to force $P$ to contain a large
standard example. The first example was studied, for a 
quite different purpose, by Howard and Trotter~\cite{bib:HowTro}. In
dual form, this poset was also studied by F\"{u}redi and 
Kahn~\cite{bib:FurKah}.

\medskip
\noindent
\textit{Example 1.}\quad
Consider a finite projective plane of order $q$.  We associate with this
geometry a poset $P$ of height $2$ with $|P|=2n=2(q^2+q+1)$.  
The minimal elements of $P$ are
the points of the geometry and the maximal elements of $P$ are the
lines.  In $P$, point $x$ is less than line $y$ when $x$ is \textit{not}
on~$y$.  It can be derived from results in \cite{bib:Blok} that if $X$ is a set
of points and $Y$ is a set of lines in a finite projective plane of
order $q$, and no point of $X$ belongs to any line in $Y$, 
then $|X||Y|\le q^3$.  It follows that if $S_d$ is
contained in $P$, then $d\le 2q^{3/2}$.  In fact, Ill\'{e}s, Sz\H{o}nyi
and Wettl~\cite{bib:IlSzWe} tightened this elementary bound
and showed that $d\le q\sqrt{q}+1$.

Now suppose that $t=\dim(P)$ and let $\cgR=\{L_1,L_2,\dots,L_t\}$ be
a realizer of $P$.  Suppose that $t< q^2+q+1-q^{3/2}$.  Then there is
a set $X$ of points, with $|X|>q^{3/2}$, so that no point
of $X$ is the top point in any linear extension in $\cgR$. (Here ``point'' refers to a point in the geometry.) Dually, 
there is a set $Y$ of lines, with $|Y|>q^{3/2}$, so that no line
in $Y$ is the lowest line in any linear extension in $\cgR$. Using
the basic property of a finite geometry that two points determine
a unique line, it follows easily that $x<y$ in $L_i$ for all $x\in X$,
$y\in Y$ and for all $i=1,2,\dots,t$. Indeed, if $x\parallel y$, then there is a linear extension in which $y<x$; but there are also a point $x'$ and a line $y'$ such that in the same linear extension $y'<y<x<x'$, which would imply $\{y,y'\}\parallel\{x,x'\}$, a contradiction.

We have shown that no point of $X$ is
on any of the lines in $Y$.  This is a contradiction, since $|X||Y|>
q^3$.  It follows that $\dim(P)\ge q^2+q+1 -q^{3/2}$, so when 
$q$ is large, $\dim(P)=(1-o(1))n$, yet if $S_d$ is a
standard example contained in $P$, then $d=O(n^{3/4})=o(n)$. 
As we will detail later, this example provides the lower bound
$f(c)=\Omega(c^{4/3})$ for the function $f(c)$ in our main theorem.

\medskip
\noindent
\textit{Example 2.}\quad
This example uses results from~\cite{bib:ErKiTr} where Erd\H{o}s, Kierstead
and Trotter study the behavior of a random poset $P$ of height~$2$.  
In this setting $P=A\cup B$, where $A$ and $B$ are $n$-element antichains.
Fix a probability $p$ (in general $p$ is a function of $n$).  Then for
each of the $n^2$ pairs $(a,b)\in A\times B$, we set $a<b$ in $P$ with
probability~$p$.  Events corresponding to distinct pairs are independent. 

The following statement is simply an extraction of more comprehensive results
proved in~\cite{bib:ErKiTr}:\quad If $p=1/2$, and $n$ is very large, 
then almost surely, the following two statements hold:
(1)~$\dim(P)> n-1000n/\log n$, and (2)~$P$ does not contain any standard 
example $S_d$ with $d\ge 100\log n$.

\section{Essential Preliminary Material}\label{sec:preliminaries}

The proof of our main theorem will require a number of well-known
inequalities in dimension theory.  These results use the following 
notation and terminology.

When $P$ is a poset, $\Min(P)$ and $\Max(P)$ denote, respectively,
the set of minimal elements and the set of maximal elements of $P$.
A subposet $D$ of $P$ is called a \textit{down set} if $y\in D$ whenever
$x\in D$ and $x>y$ in $P$.  Dually, a subposet $U$ is called an
\textit{up set} in $P$ if $y\in U$ whenever $x\in U$ and $y>x$ in $P$.
Of course, $D$ is a down set in $P$ if and only if $U=P-D$ is an up set
in $P$.  When $x$ is a point in $P$ and $Q$ is a subposet of $P$, 
we let;

\begin{enumerate}
\item $D(x,Q) = \{u\in Q:u<x$ in $P\}$;
\item $U(x,Q) = \{v\in P:v>x$ in $P\}$; and
\item $I(x,Q) = \{y\in Q-\{x\}:y\parallel x$ in $P\}$.
\end{enumerate}

When $A$ is a maximal antichain in $P$, the subposet $P-A$ is naturally
partitioned into a subposet $D_P(A)=\{x\in P-A:x<a$ for some $a\in A\}$
and a subposet $U_P(A)=\{y\in P-A:y>a$ for some $a\in A\}$.  When
no confusion may arise, we simply write $D(A)$ and $U(A)$ rather than
$D_P(A)$ and $U_P(A)$.  In discussing subposets
of $P$, we use the natural convention that when $Q$ is empty,
$\width(Q)=\dim(Q)=0$.

With this notation in hand, we list in the following theorem the
essential results we will need.

\begin{theorem}\label{thm:comprehensive}
Let $P$ be poset.  Then the following inequalities hold.
\begin{enumerate}
\item $\dim(P)\le\width(P)$.
\item If $|P|\ge 2$ and $x\in P$, then $\dim(P)\le 1+\dim(P-\{x\})$.
\item $\dim(P)\le\max\{2, 1+\width(P-\Min(P))\}$.
\item If $A$ is a maximal antichain in $P$,
then $\dim(P)\le\max\{2,|P-A|\}$.
\item If $A$ is a maximal antichain in $P$,
then $\dim(P)\le\max\{2, 1+2\width(P-A)\}$.
\item If $D$ is a down set in $P$, and $U=P-D$, then $\dim(P)\le
\dim(D)+\width(U)$.
\item If $a\in \Min(P)$, $b\in \Max(P)$ and $a\parallel b$ in $P$,
then $\dim(P)\le 1+\dim(P-\{a,b\})$.
\end{enumerate}
\end{theorem}
The first inequality is due to Dilworth~\cite{bib:Dilw}.
The second and seventh are due to Hiraguchi~\cite{bib:Hira}.
The third, fourth and fifth are due to Trotter~\cite{bib:Trot-2} 
(the fourth was discovered independently by Kimble~\cite{bib:Kimb}).
The sixth is due to Trotter and Wang~\cite{bib:TroWan}.

These inequalities have forms which can be applied to the dual of
a poset, and we will use these dual forms without comment.

\subsection{Reversible Sets and Alternating Cycles}

Let $P$ be a poset and let $\Inc(P)=\{(x,y):x\parallel y$ in $P\}$.
When $(x,y)\in \Inc(P)$ and $L$ is a linear extension of $P$, we
say $L$ reverses $(x,y)$ when $x>y$ in $L$.  More generally, we say
a family $\cgF$ of linear extensions \textit{reverses} a set $S\subseteq
\Inc(P)$ when for every $(x,y)\in S$, there is some $L\in \cgF$ which
reverses $(x,y)$. The dimension of $P$ is then the least positive
integer $d$ for which there is a family $\cgF$ of $d$ linear extensions
of $P$ which reverses $\Inc(P)$.  This reformulation of dimension 
in terms of a partition of the set of incomparable pairs was first
stated explicitly by Rabinovitch and Rival in~\cite{bib:RabRiv}. 

A set $S\subseteq \Inc(P)$ is said to be \textit{reversible} when there
is a single linear extension $L$ of $P$ which reverses every pair
in $S$.  So the dimension of a poset $P$ which is not a chain is
then the least $d$ for which $\Inc(P)$ can be partitioned into
$d$ reversible sets. 

When $k\ge2$, a sequence $\{(a_i,b_i):1\le i\le k\}$ of pairs
from $\Inc(P)$ is called an \textit{alternating cycle} when
$a_i\le b_{i+1}$ in $P$ for every $i=1,2,\dots,k$ (this requirement
is interpreted cyclically, i.e., we intend that $a_k\le b_1$ in $P$).
An alternating cycle is \textit{strict} when $a_i\le b_j$ in $P$ if
and only if $j=i+1$.  The following elementary lemma is proved (with
different notation) in~\cite{bib:TroMoo}.

\begin{lemma}\label{lem:ac}
Let $P$ be a poset and let $S\subseteq \Inc(P)$.  Then the following
statements are equivalent.

\begin{enumerate}
\item $S$ is reversible.
\item $S$ does not contain an alternating cycle.
\item $S$ does not contain a strict alternating cycle.
\end{enumerate}
\end{lemma}

In many instances, the last statement of Lemma~\ref{lem:ac} is particularly
useful, since if $\{(a_i,b_i):1\le i\le k\}$ is a strict alternating
cycle, then $\{a_1,\dots,a_k\}$ and $\{b_1,\dots,b_k\}$ are
$k$-element antichains in $P$.

The following elementary proposition was first exploited
by Hiraguchi in  proving the last inequality listed in
Theorem~\ref{thm:comprehensive}.

\begin{proposition}\label{pro:gen-mixpair}
Let $P$ be a poset and let $(x,y)\in \Inc(P)$.  
Then the set 
\[
S=\{(x,u):(x,u)\in \Inc(P)\}\cup\{(v,y):(v,y)\in\Inc(P)\}
\]
is reversible.
\end{proposition}
\begin{proof}
Since $x\parallel y$ in $P$, the set $S$ cannot contain a
strict alternating cycle.
\end{proof}

Framing dimension problems in terms of reversible sets of incomparable
pairs has been a very useful approach, and we point to the
recent papers~\cite{bib:FeTrWi} and~\cite{bib:JMMTW} as examples.

\subsection{Bipartite Posets}

In~\cite{bib:TroBog}, Trotter and Bogart defined the \textit{interval
dimension} of a poset $P$, denoted $\Idim(P)$, as 
the least positive integer $d$ for which
there are $d$ interval orders $P_1$, $P_2,\dots,P_d$ so that $x<y$ in $P$
if and only if $x<y$ in $P_i$ for $i=1,2,\dots,d$. Since a linear order
is an interval order, we always have $\Idim(P)\le\dim(P)$.  At one
extreme, it is easy to see that $\Idim(S_d)=\dim(S_d)=d$, for every
$d\ge2$. At the other extreme, the family of canonical interval orders
discussed previously show that every $d\ge1$, there is a poset with 
$\Idim(P)=1$ and $\dim(P)=d$.

In the arguments to follow, we will quickly reduce the problem
to the case where $P$ is a poset of height~$2$, and there the real
work begins. In fact, it will be useful to consider the notion
of a \textit{bipartite poset}.  This is a poset $P$ with
a partition $P=A\cup B$ where $A\subseteq\Min(P)$ and $B\subseteq
\Max(P)$. So a bipartite poset has height at most~$2$, but elements
of $P$ which are both minimal and maximal (some authors call these
elements ``loose'' points) can belong to $A$ or $B$. In fact,
an antichain can be made into a bipartite poset.  When we refer
to a subposet $Q$ of bipartite poset $P=A\cup B$, then we automatically
consider $Q= (Q\cap A)\cup (Q\cap B)$ in the bipartite form it inherits
from $P$.  Also, in the bipartite poset setting, it makes
sense to speak of the standard example $S_1=\{a_1\}\cup \{b_1\}$ with
$a_1\parallel b_1$ in $S_1$.

For the balance of this subsection, we restrict our attention to
bipartite posets.  The following two results are extracted 
from~\cite{bib:TroBog}, and we caution the reader to remember that they 
hold only in this special setting.

\begin{proposition}\label{pro:Idim-height2}
When $P=A\cup B$ is a bipartite poset,
$\Idim(P)$ is the least positive integer $d$ for which there is
a family $\cgF$ of linear extensions of $P$ so that for every
pair $(a,b)\in \Inc(P)\cap(A\times B)$, there is some $L\in\cgF$
with $a>b$ in $P$.
\end{proposition}

\begin{proposition}\label{pro:Idim-bound}
When $P=A\cup B$ is a bipartite poset,
$\Idim(P)\le\dim(P)\le 1+\Idim(P)$.
\end{proposition}

In view of the two preceding propositions,
when $P=A\cup B$ is a bipartite poset, we let 
$\Inc_0(P)=\Inc(P)\cap(A\times B)$, so that when
$\Inc_0(P)\neq\emptyset$,
$\Idim(P)$ is the least positive integer for which $\Inc_0(P)$ can be
partitioned into $d$ reversible sets.  Also, borrowing from the
terminology discussed earlier, when   
$A'\subseteq A$ and $B'\subseteq Y$, 
we will say that a linear extension $L$ \textit{reverses} $A'$ with
$B'$ when $L$ reverses all pairs  $(a,b)\in\Inc_0(P)\cap(A'\times B')$.
When $A'=\{a\}$ and $L$ reverses $A'$ with
$B'$, we will just say that $L$ reverses $a$ with $B'$.
Similarly, when $B'=\{b\}$ and $L$ reverses $A'$ with $B'$, we will
just say that $L$ reverses $A'$ with $b$.  More generally, 
we will say that a family $\cgF$ of linear extensions of $P$ \textit{reverses}
$A'$ with $B'$ when for every pair $(a,b)\in\Inc_0(P)\cap(A'\times B')$,
there is some $L\in\cgF$ with $a>b$ in $L$.
For convenience, we will just say that a family
$\cgF$ which reverses $A$ with $B$ is a \textit{reversing family} for
$P$.

Our arguments for bipartite posets will make extensive use
of the following special case of Proposition~\ref{pro:gen-mixpair}.

\begin{proposition}\label{pro:bipartite-mixpair}
Let $P=A\cup B$ be a bipartite poset.  If
$(a,b)\in \Inc_0(P)$, then
there is a linear extension $L=L(a,b,P)$ of $P$ which reverses
$a$ with $B$ and $A$ with $b$.
\end{proposition}

In the remainder of the paper, when $P=A\cup B$ is a bipartite poset, 
$(a,b)\in\Inc_0(P)$, then $L(a,b,P)$ will always denote a linear extension of 
$P$  which reverses $a$ with $B$ and $A$ with $b$.  There may be many linear
extensions which satisfying these two requirements, and in most settings,
it will not matter which one is chosen.  However, later in this
section, we will discuss a special case where we will attempt to find
a linear extension $L(a,b,P)$ which also satisfies a third requirement.

We say a subposet $Q$ of a bipartite poset $P=A\cup B$ is \textit{balanced}
when $|Q\cap A|=|Q\cap B|$.  When $m\ge 1$ and $Q$ is a balanced subposet
of $P$ with $|Q|=2m$, a labelling 
$Q=\{u_1,u_2,\dots,u_m\}\cup \{v_1,v_2,\dots,v_m\}$ of the elements
of $Q$ will be called a \textit{matching} of $Q$ when  $u_i\in A$, $v_i\in B$ and
$u_i\parallel v_i$ in $P$ for all $i=1,2,\dots,m$.

Here are two essential---although straightforward---lemmas
concerning matchings.

\begin{lemma}\label{lem:match-1}
Let $P=A\cup B$ be a bipartite poset, let $Q$ be a non-empty balanced subposet
of $P$ and let $Q=\{u_1,u_2,\dots,u_m\}\cup\{v_1,v_2,\dots,v_m\}$ be
a matching of $Q$.  Then $\Idim(P)\le m +\Idim(P-Q)$. Furthermore, if $Q$ is maximal, then $\Idim(P)\leq m$.
\end{lemma}

\begin{proof}
Let $t=\Idim(P-Q)$ and let $\{M_1,M_2,\dots,M_t\}$ be a family
of linear extensions of $P-Q$ which forms a reversing family for
$P-Q$.  For each $i=1,2,\dots,t$, let $L_i$ be a linear extension
of $P$ so that the restriction of $L_i$ to $P-Q$ is $M_i$.
Then for each $j=1,2,\dots,m$, let $L_{t+j}=L(u_j,v_j,P)$.
Then $\cgF=\{L_1,L_2,\dots,L_t,L_{t+1},\dots,L_{t+m}\}$ shows
that $\Idim(P)\le m+\Idim(P-Q)$.

Furthermore, if $Q$ is maximal, then we can take $t=0$ and the initial family empty; then $\{L_1,L_2,\ldots,L_t\}$ is a realizer. 
\end{proof}

\begin{lemma}\label{lem:save-1}
Let $P=A\cup B$ be a bipartite poset, and let $m=\min\{|A|,|B|\}$. If
$m\ge2$, then $\Idim(P) < m$ unless
$P$ contains the standard example $S_m$.
\end{lemma}

\begin{proof}
Without loss of generality, we assume $|A|\le |B|$ and
let $A=\{a_1,a_2,\dots,a_m\}$. Clearly, we may assume
that this labelling has been done so that there is no
element $b\in B$ with $b\parallel a_m$ in $P$ and
$b>a_i$ in $P$ for each $i=1,2,\dots,m-1$.  For each $i=1,2,\dots,
m-1$, let $L_i$ be a linear extension of $P$ with the following
block structure:

\[
A-\{a_i,a_m\}< I(a_i,B)\cap I(a_m,B)<a_m<I(a_i,B)\cap U(a_m,B)<a_i<U(a_i,B).
\]
Clearly, the family $\{L_1,L_2,\dots,L_{m-1}\}$ shows that
$\Idim(P)\le m-1$.
\end{proof} 

The next lemma is more substantive and more technical in
nature.  However, this result will prove to be a key detail in our proof.
Also, it is one of two new inequalities prompted by the
general problem investigated in this paper---the second such result
will be presented in the next section.
We alert the reader that we will be discussing linear
extensions of the form $L(a,b,P)$ where we search for such
an extension which also reverses a pair $(a',b')\in \Inc_0(P)$ with
$a\neq a'$ and $b\neq b'$.  

\begin{lemma}\label{lem:two-standard}
Let $s\ge1$ and $t=5s$. Then let $P=A\cup B$ be a balanced bipartite poset with
$|P|=4t$. Suppose that
$P$ can be partitioned into two disjoint balanced subposets $T$ and $T'$, each of which is a copy of the
standard example $S_t$.  Also, let
$T=\{a_1,a_2,\dots,a_t\}\cup \{b_1,b_2,\dots,b_t\}$ and 
$T'=\{w_1,w_2,\dots,w_t\}\cup\{z_1,z_2,\dots,z_t\}$ be matchings.
If the subposet $\{a_i,b_i,w_i,z_i\}$ is not $S_2$, for
each $i=1,2,\dots,t$,  then $\Idim(P)\le 9s$.
\end{lemma}

\begin{proof}
Set $A_0=\{a_1,a_2,\dots,a_t\}$ and $W_0=\{w_1,w_2,\dots,w_t\}$ so
that $A=A_0\cup W_0$.  Also, set $B_0=\{b_1,b_2,\dots,b_t\}$ and
$Z_0=\{z_1,z_2,\dots,z_t\}$ so that $B=B_0\cup Z_0$.

Before proceeding with the proof, we pause to make two comments.
First, we have the trivial upper bound $\Idim(P)\le |A|= 10s$, so the
purpose of the lemma is just to lower this upper bound down to
$9s$.  Second, when a family $\cgF$ of linear extensions of $P$ 
is reversing, it must reverse the ``vertical'' matched pairs 
$\{(a_i,b_i):1\le i\le t\} \cup\{(w_i,z_i):1\le i\le t\}$ as well 
as the ``diagonal'' pairs in 
\[
\bigl(\Inc_0(P)\cap(A_0\times Z_0)\bigr)\cup 
\bigl(\Inc_0(P)\cap(W_0\times B_0)\bigr).
\]

Let $s_1$ be the largest non-negative integer for which there are
sets $C=\{i_1,\dots,i_{s_1}\}$ and $D=\{j_1,\dots,j_{s_1}\}$,
each of which are $s_1$-element subsets of $\{1,2,\dots,t\}$ so that
for each  $k=1,2,\dots,s_1$, $a_{i_k}\parallel z_{j_k}$ and
$b_{i_k}\parallel w_{j_k}$ in $P$.  Note that we may have
$s_1=0$.   Regardless, we show that $\Idim(P)\le 10s-s_1$.
There is nothing to prove if $s_1=0$, so we consider the case
where $s_1>0$.  Let $Q_1$ be the following subposet of $P$:

\[
Q_1=\cup\bigl\{\{a_{i_k},b_{i_k},w_{j_k},z_{j_k}\}:1\le k\le s_1\bigr\}.
\]

By Lemma~\ref{lem:match-1}, we have 
$\Idim(P)\le 10s -2s_1+\Idim(Q_1)$. To show that $\Idim(P)\le 10s-s_1$,
we need only show that $\Idim(Q_1)\le s_1$.  However, this follows
from the fact that for every $k=1,2,\dots,s_1$, there is a linear extension
$L(a_{i_1},b_{i_1},Q_1)$ which also reverses $(w_{i_1},z_{i_1})$.

If $s_1\ge s$, then it follows that $\Idim(P)\le 9s$.  So for the
balance of the argument, we will assume that $s_1<s$. 

Let $E=C\cup D$. Then $|E| < 2s$, so we may assume that
after a relabelling of subscripts that $E\cap \{1,2,\dots,3s\}=\emptyset$.
In particular, this implies that for all $i,j=1,2,\dots,3s$, if
$a_i\parallel z_j$ in $P$, then $w_j<b_i$ in $P$.  Also, if 
$w_j\parallel b_i$ in $P$, then $a_i<z_j$ in $P$.  In particular, it
implies that for each $i=1,2,\dots,3s$, exactly one of the following
two statements applies:

\begin{enumerate}
\item $a_i\parallel z_i$ and $w_i<b_i$ in $P$.
\item $w_i\parallel b_i$ and $a_i<z_i$ in $P$.
\end{enumerate}

Let $Q_2$ denote the subposet determined by all elements with 
subscripts (after the relabelling) in the range
$1\le i\le 3s$.  To complete the proof of the lemma,
we need only show that $\Idim(Q_2)\le 5s$.  

We define an auxiliary directed graph $G$ whose vertex set is
$\{1,2,\dots,3s\}$.  In $G$ we have a directed edge $(i,j)$
when $i$ and $j$ are distinct integers in
$\{1,2,\dots,3s\}$ and one of the following two conditions
applies:

\begin{enumerate}
\item $a_i\parallel z_j$ and $b_i\parallel w_i$ in $P$.
\item $b_i\parallel w_j$ and $a_i\parallel z_i$ in $P$.
\end{enumerate}
We note that these conditions are mutually exclusive. We also note that we
may have edges $(i,j)$ and $(j,i)$ simultaneously.

Now let $M$ be a maximal matching in the auxiliary
graph $G$ and suppose that $M$ consists of $r$ edges, where
$0\le r\le 3s/2$.  Let $Q_3$ be the subposet of $Q_2$ determined 
by the elements whose subscripts are in the maximal matching $M$.
Using Lemma~\ref{lem:match-1}, and that $|Q_3|=4r$, we have $\Idim(Q_2)\le 6s-4r+\Idim(Q_3)$.  

We show that $\Idim(Q_3)\le 3r$. There is nothing to prove if
$r=0$ and $Q_3=\emptyset$, so we assume $r>0$.
We construct a family $\cgF_3$ of linear extensions of $Q_3$ by
the following rule: For each edge $(i,j)$ in the matching $M$, we add
to $\cgF_3$ three linear extensions determined as follows:

\begin{enumerate}
\item If the first condition above applies, i.e.~$a_i\parallel z_j$ and $b_i\parallel w_i$ in
$Q_2$, then we add $L(a_i,z_j,Q_2)$, $L(w_i,b_i,Q_2)$ and $L(a_j,b_j,Q_2)$
to $\cgF_3$.
\item If the second condition above applies, i.e.~$b_i\parallel w_j$ and $a_i\parallel z_i$ in
$Q_2$, then we add $L(w_j,b_i,Q_2)$, $L(a_i,z_i,Q_2)$ and $L(a_j,b_j,Q_2)$
to $\cgF_3$.
\end{enumerate}

We claim that $\cgF_3$ is reversing for $Q_3$. Indeed, if $(i,j)$ is in $M$, then any incomparable pair involving $a_i$, $a_j$, $b_i$, or $b_j$ are reversed, and depending on which condition is satisfied, either $z_i$ and $w_j$, or $z_j$ and $w_i$ are reversed with all elements they are incomparable with. Now it is elementary to verify that all vertical and diagonal pairs are reversed in $\cgF_3$.

Furthermore,
$|\cgF_3|=3r$.  If $r\ge s$, this implies that 
\[
\Idim(Q_2)\le6s-4r+\Idim(Q_3)\le (6s-4r)+3r=6s-r\le 5s.
\]
Accordingly, we may assume that $r<s$.  

Now let $Q_4=Q_2-Q_3$, i.e., $Q_4$ is the subposet of $Q_2$ determined
by elements whose subscripts are \textit{not} in the maximal matching $M$.

Then $\Idim(Q_2)\le 4r+\Idim(Q_4)$.  Again, we note that this
inequality holds even if $r=0$.
We build a family $\cgF_4$ of linear extensions of 
$Q_4$ as follows.  For each vertex
$i$ which is not covered by the maximal matching $M$, we
add $L(a_i,z_i,Q_4)$ to $\cgF_4$ if $a_i\parallel z_i$ in $Q_4$. On the
other hand, if $a_i<z_i$ in $Q_4$, then $w_i\parallel b_i$ in $Q_4$ and
in this case, we add $L(w_i,b_i,Q_4)$ to $\cgF_4$.  

We now show that $\cgF_4$ is a reversing
family for $Q_4$. First note that the vertical pair $(a_i,b_i)$ gets reversed, because either $L(a_i,z_i,Q_4)\in\cgF_4$ or $L(w_i,b_i,Q_4)\in\cgF_4$. Similarly the vertical pairs $(w_i,b_i)$ are reversed. Now
consider a diagonal pair 
$(u,v)\in\Inc_0(Q_4)$. There are integers $i$ and $j$ 
with $1\le i,j\le 3s$, for which $u\in\{a_i,w_i\}$ and
$v\in\{z_j,b_j\}$.  We may assume $i\neq j$, for otherwise $L(u,v,Q_4)\in\cgF_4$. 
If $u=a_i$ and $v=z_j$, then the
maximality of $M$ implies
that $G$ has neither an $(i,j)$ nor a $(j,i)$ edge, so
$a_i\parallel z_j$ implies $w_i<b_i$, and if also $a_i<z_i$,
then $\{w_i,b_i,a_i,z_i\}$ would form an $S_2$.
Hence $a_i\parallel z_i$ in $Q_4$. (Symmetric argument shows
$a_j\parallel z_j$ in $Q_4$.) So $(u,v)$ is reversed in 
$L(a_i,z_i,Q_4)$ and in $L(a_j,z_j,Q_4)$, both of which belong
to $\cgF_4$. The case, when $u=w_i$ and $v=b_j$ is handled similarly.

This completes the proof that $\cgF_4$ is a reversing family
for $Q_4$.  Furthermore, it is clear that $|\cgF_4|=3s-2r$, so that 
\[
\Idim(Q_2)\le 4r+\Idim(Q_4)\le 4r+(3s-2r)= 3s+2r < 5s.
\]

This completes the proof.
\end{proof}

As we bring this subsection to a close, we remind the reader that
we will no longer be restricting our attention to bipartite posets.

\subsection{A New Inequality}

The following lemma will be an essential tool for reducing the
problem to posets of height~$2$.  Its formulation
was motivated entirely by the problem at hand; however, the ideas behind
the proof are a relatively straightforward extension of techniques first
introduced by Trotter and Monroe in~\cite{bib:TroMon}.

\begin{lemma}\label{lem:TroMon}
Let $A$ be a maximal antichain in a poset $P$ which is
not an antichain. If $X=D(A)$ and $Y=U(A)$ are both
antichains, $|X|=s$ and $|Y|=s+t$ where $s,t\ge 0$, 
then $\dim(P)\le 1+t+\lceil 4s/3\rceil$.
\end{lemma}

\begin{proof} 
We argue by contradiction.  First, we assume the lemma is false, and
let $P$ be a counterexample with $|P|$ as small as possible.
Suppose first that $s=|X|=0$.  Then $A=\Min(P)$.
Furthermore, since $P$ is not
an antichain, $U(A)\neq\emptyset$.  Then, by Theorem~\ref{thm:comprehensive} (3),
$\dim(P)\le 1+\width(Y)\le 1+t$.  The contradiction shows that $s>0$.

Now suppose that $t>0$. Let $y\in Y$ and consider the poset
$Q=P-\{y\}$. Since $P$ is a minimum size counterexample, we know
that $\dim(Q)\le t+\lceil4s/3\rceil$, but this implies that
$\dim(P)\le 1+t+\lceil4s/3\rceil$.  The contradiction forces $t=0$.

Now suppose $s=1$.  Then $|P-A|=2$ so by Theorem~\ref{thm:comprehensive} (4), $\dim(P)\le 2< 1+\lceil 4/3\rceil = 3$.
The contradiction shows $s>1$.  Now suppose $s=2$.  Then
$|P-A|=4$, so $\dim(P)\le 4=1+\lceil 8/3\rceil$. 
The contradiction shows $s\ge 3$.

Next, we observe that we must have $x<y$ in $P$ for all $x\in X$ and
$y\in Y$.  For if there was an incomparable pair
$(x,y)\in X\times Y$, we could remove $x$ and $y$ and decrease the
dimension by at most~$1$. This would again produce a counterexample of
smaller size.

We now attempt to construct a realizer $\cgR$ of $P$ with 
$|\cgR|=1+\lceil 4s/3\rceil$.  We first consider the
case where $s\equiv0 \mod3$, as the other two residue classes are easy 
modifications of this base case.   Furthermore, the first non-trivial
case is $s=3$.

Label the elements of $X$ as $\{x_1,\dots,x_s\}$ and
the elements of $Y$ as $\{y_1,\dots,y_s\}$.  
Set $r=s/3$.  For each $j=1,2,\dots,r$, we construct 
four linear extensions $L_{4j-3}$, $L_{4j-2}$,
$L_{4j-1}$ and $L_{4j}$.  These four extensions will
focus on the elements of 
\[
\{x_{3j-2},x_{3j-1},x_{3j}\}\cup \{y_{3j-2}, y_{3j-1},y_{3j}\}.
\]
In each of the four
extensions, $x_{3j-2}$, $x_{3j-1}$ and $x_{3j}$ will be
the three highest elements of $X$. Also, $y_{3j-2}$,
$y_{3j-1}$ and $y_{3j}$ will be the three lowest elements
of $Y$.  Furthermore, the restriction of the four
extensions to these six elements will be:

\begin{align*}
&x_{3j-1} < x_{3j} < x_{3j-2} < y_{3j-2} < y_{3j-1} < y_{3j} \text{  in $L_{4j-3}$}.\\
&x_{3j} < x_{3j-1} < x_{3j-2} < y_{3j-1} < y_{3j-2} < y_{3j} \text{  in $L_{4j-2}$}.\\
&x_{3j-2} < x_{3j} < x_{3j-1} < y_{3j} < y_{3j-2} < y_{3j-1} \text{  in $L_{4j-1}$}.\\
&x_{3j-2} < x_{3j-1} < x_{3j} < y_{3j} < y_{3j-1} < y_{3j-2} \text{  in $L_{4j}$}.
\end{align*}

In each of these four extensions, we have seven places (blocks) into
which elements of $A$ can be placed.  But recall that our goal
is to reverse pairs of the form $(a,y)$ where
$a\in A$ and $y\in\{y_{3j-2},y_{3j-1},y_{3j}\}$ as well
as pairs of the form $(x,a)$ where $x\in\{x_{3j-2},x_{3j-1},x_{3j}\}$.
So in the discussion to follow, when we refer to an incomparable
pair $(a,y)$, we intend that $a\in A$ and $y\in\{y_{3j-2},y_{3j-1},y_{3j}\}$.
An analogous remark applies to pairs $(x,a)$.

First, let  $a\in A$ and suppose that by placing $a$ in the highest
possible block in $L_{4j-3}$, we have succeeded in reversing
all pairs (if any) of the form $(a,y)$.
Then $a$ can be pushed down into the lowest 
possible blocks in $L_{4j-2}$, $L_{4j-1}$ and $L_{4j}$ and
we will certainly reverse all pairs of the form $(x,a)$.
An analogous statement holds if we could put $a$ in the highest possible
block in $L_{4j-2}$ and reverse all pairs (if any) of the form
$(a,y)$. 

Dually, suppose that by placing $a$ as low as possible in 
$L_{4j}$, we have succeeded in reversing all pairs $(x,a)$.  Then
we could push $a$ up in $L_{4j-3}$, $L_{4j-2}$ and $L_{4j-1}$ and we
will have certainly reversed all pairs of the form $(a,y)$. 
An analogous statement holds if we could put $a$ in the lowest possible
block in $L_{4j-1}$ and reverse all pairs (if any) of the form
$(x,a)$.

There are three cases left to consider:
\begin{enumerate}
\item $a\parallel y_{3j-2}$ in $P$, $a\parallel y_{3j}$ in $P$ and
$a<y_{3j-1}$ in $P$.
\item $a\parallel y_{3j-1}$ in $P$, $a\parallel y_{3j}$ in $P$ and
$a<y_{3j-2}$ in $P$.
\item $a < y_{3j-2}$ in $P$, $a<y_{3j-1}$ in $P$ and
$a\parallel y_{3j}$ in $P$.
\end{enumerate}

In the first case, we push $a$ up in $L_{4j-1}$ and down in the
other three linear extensions in our group.
(Note that we can not have $a\parallel x_{3j-1}$ and $a>x_{3j-2}$, because otherwise we would have pushed down $a$ in $L_{4j-1}$ or $L_{4j}$ as explained above. So $a\parallel x_{3j-1}$ implies $a\parallel x_{3j-2}$
and so it will be pushed under $x_{3j-1}$ in $L_{4j-2}$.)
In the second,
we push $a$ up in $L_{4j}$ and down in the other three.
Finally, in the third case, we push $a$ up in $L_{4j}$ and
down in the other three.
%Old text: This works unless $a\parallel x_{3j-1}$ in
%$P$, $a>x_{3j-2}$ in $P$ and $a>x_{3j}$ in $P$.  Now we push
%$a$ up in $L_{4j}$ and down in the other three.
%
%We don't need this. This case is covered when we push $a$
%down in L_{4j-1}, and it reverses all pairs $(x,a)$:

These remarks complete the proof in the case when $s\equiv 0\mod 3$.
When $s\equiv 1 \mod3$ and $s=3r+1$, we  note that $\lceil 4s/3\rceil=
4r+2$.  So to the family constructed above, we add two additional
linear extensions each having $x_{s}$ as the highest element of
$X$ and $y_s$ as the lowest element of $Y$.  Now, elements
of $A$  are pushed down in the first of the two new linear
extensions and pushed up in the second.

When $s\equiv 2\mod3$ and $s=3r+2$, we  note that $\lceil 4s/3\rceil=
4r+3$.  So to the family of size $3r$ constructed above, we add three additional
linear extensions. Each has $x_{s-1}$ and $x_s$ as the highest elements
of $X$ and $y_{s-1}$ and $y_s$ as the lowest elements of $Y$.
The first two have $x_{s-1}<x_{s}$ while the third has 
$x_s<x_{s-1}$.  However, only the first has $y_{s-1}<y_s$ with
$y_s<y_{s-1}$ in both the second and the third.
It is easy to see how to appropriately position elements of $A$
in these three new linear extensions, and these observations
complete the proof of the lemma.
\end{proof} 

We comment that when $s\ge1$, we can actually prove that 
$\dim(P)\le t+\lceil 4s/3\rceil$.  We do not include the proof as 
the technical details are formidable, and the minor improvement is 
not central to the results of this paper.  However, when $t=0$, the 
resulting inequality $\dim(P)\le\lceil 4s/3\rceil$ is tight, as evidenced 
by examples constructed in~\cite{bib:TroMon}.

\section{Proof of the Main Theorem}\label{sec:proof}

For the readers convenience, we restate here the theorem we are
about to prove:

\medskip
\noindent
\textbf{Theorem.}\quad
For every positive integer $c$, there is an integer $f(c)=O(c^2)$ so that if 
$n>10f(c)$ and $P$ is a poset with $|P|\le 2n+1$ and $\dim(P)\ge n-c$, 
then $P$ contains a standard example $S_d$ with $d\ge n-f(c)$.

\begin{proof}
Let $c$ be a positive integer.  Then set $s=41c$, $t=5s$,
and $f(c)=17ct$.  We note that $f(c)=3485c^2$.

Let $n>10f(c)$ (this bound is generous) and 
let $P$ be a poset with $|P|\le 2n+1$ and $\dim(P)\ge n-c$.  
We will show that $P$ contains a standard example $S_d$ with $d\ge
n-f(c)$.  Clearly, we may assume that $|P|=2n+1$, as otherwise
we can just add loose points. 

Let $A$ be a maximum antichain in $P$.  Since $n-c\le\dim(P)\le\width(P)$,
we know $|A|\ge n-c$.  Since $\dim(P)\le |P-A|$, we also know that
$|A|\le n+c+1$. Let $D=P-U(A)$. By Theorem~\ref{thm:comprehensive},
\[
\dim(P)\le \dim(D)+\width(U(A))\le 1+\width(D(A))+\width(U(A)),
\]  
so we may choose antichains $X\subseteq D(A)$ and $Y\subseteq U(A)$ so that 
$|X|+|Y|=n-c-1$. We observe that since $A$ is a maximal antichain in $P$, in
the subposet $P_0=A\cup X\cup Y$,  $X=D_{P_0}(A)$ and $Y=U_{P_0}(A)$.  

Without loss of generality, we may assume that $|X|\le |Y|$.  
Set $\sigma=|X|$ and $|Y|=\sigma+\tau$ where $\tau\ge 0$.  From
Lemma~\ref{lem:TroMon}, we know that $\dim(P_0)\le 
1+\tau+\lceil 4\sigma/3\rceil\le 2+\tau+4\sigma/3$.  On the other hand, since
$|A|\ge n-c$, we know that there are at most $2c+2$ points of $P$
which do not belong to $P_0$. Therefore, $\dim(P_0)\ge n-c - (2c+2)
=n-(3c+2)$.  Since $n-c-1=|X|+|Y|=2\sigma+\tau$,
so that $n=c+1+2\sigma+\tau$, it follows
that 
\[
(c+1+2\sigma+\tau)-(3c+2)\le \dim(P_0)\le 2+\tau+4\sigma/3.
\]
This implies that $2\sigma/3\le 2c+3$, so that $\sigma\le 3c+4$.

We now focus on the bipartite poset $P_1=A\cup Y$.  Since
$\dim(P_0)\ge n-(3c+2)$ and $\sigma\le 3c+4$, we know
$\dim(P_1)\ge n- (6c+6)$. 

In order to be consistent with the material developed
in the preceding section for bipartite posets, we relabel the
set $Y$ as $B$ and reuse (in the computer science tradition) 
the symbol $P$ for the bipartite poset $A\cup B$.
With its updated definition, we know $\dim(P)\ge n-(6c+6)$, so that
$\Idim(P)\ge n-(6c+7)\ge n-13c$.

Now let $d$ be the size of the largest standard example contained in $P$.
Choose a copy $T$ of $S_d$ in $P$ with minimal elements $A_0=
\{a_1,a_2,\dots,a_d\}\subseteq A$ and maximal elements $B_0=\{b_1,b_2,
\dots,b_d\}\subseteq B$. Of course, we also intend that $a_i\parallel
b_i$ in $P$ for $i=1,2,\dots,d$.  If $d\ge n-f(c)$, the conclusion
of the theorem has been established. So we will assume that
$d<n-f(c)$ and argue to a contradiction. 

Since $f(c)+1\le n-d$ and $d+|A-A_0|=|A|\ge n-c$, we see that
$|A-A_0|\ge f(c)-c+1$.
Let $A-A_0=A_1\cup A_2\cup
\dots\cup A_{16c}$ be any partition of $A-A_0$ into parts as equal in
size as division will allow.  For each $i=1,2,\dots, 16c$, let
$P_i$ now denote the bipartite subposet $A_i\cup (B-B_0)$.
If $\Idim(P_i)<|A_i|$ for each $i=1,2,\dots,16c$, then
$\Idim(P-T)\le |A-A_0|-16c$.  Since $|A|=d+|A-A_0|$ and 
$|A|\le n+c+1\le n+2c$, using Lemma~\ref{lem:match-1} we get $\Idim(P)\le n-14c$, which is false.

After a relabelling, we may assume that $\Idim(P_1)=|A_1|$.
Therefore, by Lemma~\ref{lem:save-1}, $P-T$ contains the standard example whose dimension
is $|A_1|$.  We know that $|A_1|\ge
\lfloor (f(c) -c + 1)/16c\rfloor$, so we can safely say
$|A_1|\ge f(c)/17c=t$.  Note also that $t=5s$.  Choose  
a copy $T'$ of $S_t$ contained in $P-T$ and label the elements
of $T'$ as $W_0=\{w_1,w_2,\dots,w_t\}$ and
$Z_0=\{z_1,z_2,\dots,z_t\}$ so that for all $i,j=1,2,\dots,t$,
$w_i<z_j$ in $P$ if and only if $i\neq j$.

We associate with the bipartite subposet $(A_0\cup B_0)\cup (W_0\cup Z_0)$ 
an auxiliary graph $G$ which is a bipartite graph.   
The graph $G$ has vertex set $U\cup V$, where
$U=\{u_1,\dots,u_t\}$ and $V=\{v_1,\dots,v_d\}$.
In $G$, we have an edge $u_iv_j$ when the subposet $\{w_i,z_i,a_j,b_j\}$
is \textit{not} the standard example $S_2$.

\medskip
\noindent
\textbf{Claim 1.}\quad
In the bipartite graph $G$, there is a graph matching from $U$ to $V$,
i.e., there is a $1$--$1$ function $g:U\rightarrow V$ so that
$g(u)$ is a neighbor of $u$ for every $u\in U$.

\begin{proof}
We use Hall's theorem.  For each subset $S\subseteq U$, let $N_G(S)$ be the
subset of $V$ consisting of all vertices in $V$ adjacent in $G$ to one
or more vertices in $S$.  If the claim is false, then there is 
a set $S\subseteq U$ with $|S|>|N_G(S)|$.  However, if we remove from 
$T$ all pairs of the form $\{a_i,b_i\}$ with $v_i\in N(S)$ and replace them with
the pairs $\{w_j,z_j\}$ with $u_j\in S$, we obtain a standard example whose
dimension is $d-|N(S)|+|S|$ which is larger than $d$.  
The contradiction completes the proof of the claim.
\end{proof}

Without loss of generality, we may assume that the pairs in $A_0\cup B_0$
have been labelled so that $g(u_i)=v_i$ for all $i=1,2,\dots,t$, i.e.,
the subposet $\{a_i,b_i,w_i,z_i\}$ is not $S_2$.  Then let
$Q$ denote the bipartite poset consisting of all elements of 
$(A_0\cup W_0)\cup (B_0\cup Z_0)$ with subscripts at most $t$.
Then let $q$ be the largest integer so that there
is a balanced subposet $Q'$ of $P-Q$ with $|Q'|=2q$ so that
$Q'$ admits a matching.  Then let $P'$ be the bipartite poset
formed by $Q\cup Q'$, and note that $Q\cup Q'$ is a maximal matching. Using Lemma~\ref{lem:match-1}, we have $n-13c\le \Idim(P)\le 2t+q$,
so then $(2n+1)-(4t+2q)\le 26c+1\le 27c$, and so we conclude that there are
at most $27c$ points of $P$ which do not belong to $P'$.  It follows that
$\Idim(P')\ge (n-13c)-27c=n-40c$.  

On the other hand, $\Idim(P')\le q+\Idim(Q)$. Furthermore,
since $t=5s$, we know from Lemma~\ref{lem:two-standard}
that $\Idim(Q)\le 9s=2t-s$.  It follows that

\[
(2t+q)-40c\le n-40c\le \Idim(P')\le q+(2t-s).
\]
This implies that $s\le 40c$ which is false, since $s=41c$.
The contradiction completes the proof.
\end{proof}

\section{Fractional Dimension}\label{sec:fracdim}

The concept of \textit{fractional dimension} was introduced
by Brightwell and Scheinerman in~\cite{bib:BriSch}, but we elect
to use the alternative formulation of this parameter given
by Bir\'o, Hamburger and P\'or in~\cite{bib:BiHaPo-2}.  Let 
$\{L_1,L_2,\dots,L_t\}$ be the family of \textit{all} linear
extensions of a poset $P$.  Then the fractional dimension of
$P$, denoted $\fdim(P)$ (some authors use the notation $\fd(P)$),
is the least positive real number $d$ for which there are
non-negative real numbers $\{\alpha_i:1\le i\le t\}$ so that
(1)~$\sum_{i=1}^t \alpha_i = d$; and (2)~for every pair $(x,y)$ with
$x\parallel y$ in $P$, $\sum\{\alpha_i:1\le i\le t, x>y$ in
$L_i\}\ge 1$.

For fractional dimension, we have the following 
inequalities, all due to Brightwell and Scheinerman~\cite{bib:BriSch}.

\begin{theorem}\label{thm:fd-comprehensive}
Let $P$ be poset.  Then the following inequalities hold.
\begin{enumerate}
\item $\fdim(P)\le\dim(P)$.
\item If $x\in P$, then $\fdim(P)\le 1+\fdim(P-\{x\})$.
\item If $a\in \Min(P)$, $b\in \Max(P)$ and $a\parallel b$ in $P$,
then $\fdim(P)\le 1+\fdim(P-\{a,b\})$.
\end{enumerate}
\end{theorem}

For a bipartite poset $P=A\cup B$, there is a
natural fractional dimension analogue of the inequality
$\dim(P)\le 1+\Idim(P)$. We let $\Idim^*(P)$ be the least $d$
so that there are 
non-negative real numbers $\{\alpha_i:1\le i\le t\}$ so that
(1)~$\sum_{i=1}^t \alpha_i = d$; and (2)~for every pair $(a,b)\in 
A\times B$ with $a\parallel b$ in $P$, $\sum\{\alpha_i:1\le i\le t,
a>b$ in $L_i\}\ge 1$. Clearly $\Idim^*(P)\leq \Idim(P)$; in fact $\Idim^*(P)$ may be zero.

\begin{proposition}\label{pro:fracdim}
For a bipartite poset $P=A\cup B$,
$\fdim(P)\le 2+\Idim^*(P)$.
\end{proposition}
\begin{proof}
Let $\{\alpha_i:1\le i\le t\}$ be a set of non-negative
weights witnessing the value of $\Idim^*(P)$.  Then
let $L$ and $L'$ be linear extensions of $P$ with $A<B$ in $L$,
$A<B$ in $L'$, $L(A)$ is the dual of $L'(A)$ and $L(B)$ is
the dual of $L'(B)$.  Then increase the weights of $L$
and $L'$ by~$1$.  The resulting values show $\fdim(P)\le
\Idim^*(P)+2$.
\end{proof}

We will need the following trivial consequence of Lemma~\ref{lem:match-1}.

\begin{lemma}\label{lem:match-2}
Let $P=A\cup B$ be a bipartite poset, and let $Q$ be
a maximal matching with $m$ minimal and $m$ maximal elements in $P$. Then $\Idim^*(P)\leq m$.
\end{lemma}

A trivial consequence of Lemma~\ref{lem:TroMon}
is the following version for fractional dimension.

\begin{lemma}\label{lem:TroMon-fdim}
Let $A$ be a maximal antichain in a poset $P$ which is
not an antichain.  If $X=D(A)$ and $Y=U(A)$ are antichains,
$|X|=s$ and $|Y|=s+t$ where $t\ge 0$, 
then $\fdim(P)\le 1+t+\lceil 4s/3\rceil$.
\end{lemma}

We next turn our attention to developing analogous versions
of Theorem~\ref{thm:main-dim} for fractional dimension.
We start with the bipartite version.

\begin{theorem}\label{thm:fdim-bipartite}
For every positive integer $c$, if $n>10(5c+12)$, and 
$P=A\cup B$ is a bipartite poset with
$|P|\le 2n+1$ and $\fdim(P)\ge n-c$,
then $P$ contains a standard example $S_d$ with $d\ge n-(5c+12)$.
\end{theorem}

\begin{proof}
We will assume $|P|=2n+1$.  Otherwise add loose points which
cannot decrease the fractional dimension.
In presenting the proof, we will find it convenient to use
graph theoretic terminology for the bipartite graph $G$ whose
vertex set is $A\cup B$ with $G$ containing an edge $(a,b)$
when $(a,b)\in A\times B$ and $a\parallel b$ in $P$.  In particular,
paths and cycles in $G$ will play an important role in our proof.

Next, we will identify a set  of linear extensions of $P$ and
assign positive weights to these extensions. All other linear extensions
will be assigned weight~$0$.

First, if $G$ is acyclic, set $s=0$ and $Q_1=\emptyset$.  If
$G$ is not acyclic, let $s$ be the largest integer for which
there is a balanced subposet $Q_1$ of $P$ so that $|Q_1|=2s$ and
$Q_1$ is the union of disjoint cycles.  Note that we do not
require that the cycles be induced.  For each edge $(a,b)$
which is one of the edges on one of the cycles, we choose
a linear extension $L(a,b,P)$ reversing $a$ with $B$ and
$A$ with $b$ and assign it weight~$1/3$.

We note that $P-Q_1$ is acyclic.  If $ P-Q_1$ does not contain
a path on $4$ vertices, we set $r=0$ and $Q_2=\emptyset$; otherwise,
let $r$ be the largest integer for which there
is a subposet $Q_2$ of $P-Q_1$ so that $|Q_2|=4r$ and
$Q_2$ has a matching $\{u_1,u_2,\dots,u_{2r}\}\cup
\{v_1,v_2,\dots,v_{2r}\}$ so that for each $i=1,2,\dots,r$,
$u_{2i-1}\parallel v_{2i}$ in $P$.  Note that the maximality
of $s$ implies that $u_{2i}<v_{2i-1}$ in $P$ for each $i=1,2,\dots,r$.  
As before, for each $i=1,2,\dots,r$, we choose three linear extensions
$L(u_{2i-1},v_{2i-1},P)$, $L(u_{2i},v_{2i},P)$ and $L(u_{2i-1},v_{2i},P)$,
but now we assign weight~$1/2$ to each of them.

If there are no edges in $P-(Q_1\cup Q_2)$, set $d=0$ and $Q_3=\emptyset$;
otherwise let $d$ be the largest positive integer for which there is
a balanced $2d$-element subposet $Q_3$ in $P-(Q_1 \cup Q_2)$ with
a matching $\{a_1,a_2,\dots,a_d\}\cup\{b_1,b_2,\dots,b_d\}$.  
In this case, when $d\ge2$, we note that $Q_3$ is the standard example
$S_d$.  For each $i=1,2,\dots,d$, we choose a linear extension
$L(a_i,b_i,P)$ and assign it weight~$1$.

Set $Q_4=P-(Q_1\cup Q_2\cup Q_3)$ and let $q=|Q_4|$.  We note that if $(a,b)\in 
A\times B$, with $a,b\in Q_4$, $a<b$ in $P$. If $a\in
Q_4\cap A$, choose a linear extension $L(a,B,P)$ and assign it
weight~$1/2$. Similarly, for each
$b\in Q_4\cap B$, choose a linear extension $L(A,b,P)$ and 
assign it weight~$1/2$.

Let $(a,b)\in A\times B$ be an incomparable pair, and let $w$ be the sum of weights of linear extensions in which $(a,b)$ is reversed. If $a\in Q_3$ or $b\in Q_3$, then $w\geq 1$. If $a,b\in Q_1$, then $w\geq 1/3+1/3+1/3$. If $a\in Q_1$ or $b\in Q_1$, but not both, then $w\geq 1/3+1/3+1/2$. In all other cases, we have $a,b\in Q_2\cup Q_4$, and then $w\geq 1/2+1/2$. In all cases $w\geq 1$.

Let $t$ denote the sum of all the weights we have assigned. The argument above shows that $\Idim^*(P)\leq t$.

It follows that:

\begin{equation}\label{eqn:bound0}
n-(c+2)\le \Idim^*(P)\le t= 2s/3+3r/2+d+q/2.
\end{equation}

Recall that $2n+1=|P|=|Q_1|+|Q_2|+|Q_3|+|Q_4|=2s+4r+2d+q$, so by the previous inequality,
\[
2s+4r+2d+q-1-2(c+2)=
2n-2(c+2)
\leq 4s/3+3r+2d+q,
\]
hence $2s/3+r\leq 2c+5$.

Notice that $Q_1\cup Q_2\cup Q_3$ admits a maximal matching, so by Lemma~\ref{lem:match-2}, we get
\[
n-(c+2)\leq\Idim^*(P)\leq s+2r+d.
\]
Similarly as above, from this we conclude $q\leq 2c+5$.

We return to (\ref{eqn:bound0}), and rewrite it to get
\begin{equation}\label{eqn:bound}
d\ge n - (c+2) - (2s/3+3r/2+q/2).
\end{equation}

Considering the previously proven inequalities $2s/3+r\le 2c+5$ and $q\le 2c+5$,
inequality~\ref{eqn:bound} is weakest when $s=0$, $r=2c+5$ and
$q=2c+5$.  With these values, it becomes:

\[
d\ge n - (c+2) -(4c+10)= n-(5c+12).
\]
\end{proof}
 
Next, we present the analogous version for general posets.

\begin{theorem}\label{thm:main-fdim}
For every positive integer $c$, 
if $n>10(30c+52)$ and $P$ is any poset 
with $|P|\le 2n+1$ and $\fdim(P)\ge n-c$,
then $P$ contains a standard example $S_d$ with $d\ge n-(30c+52)$
\end{theorem}

\begin{proof}
Using the inequalities in Theorem~\ref{thm:fd-comprehensive} and
following along lines from the proof of Theorem~\ref{thm:main-dim},
we first obtain a subposet $P_0$ consisting of an antichain $A$, which
is maximum in $P$, and two other antichains $X\subseteq D(A)$ and
$Y\subseteq U(A)$ with $\fdim(P_0)\ge n-(3c+2)$.  We now use the
inequality $\fdim(P_0)\le  2+t+4s/3$ to conclude that $s\le 3c+4$.
This implies the bipartite subposet $A\cup Y$ has fractional
dimension at least $n-(6c+6)$.  After the relabelling, we have
a bipartite poset $P=A\cup B$ with $\Idim^*(P)\ge n-(6c+8)$.

From the preceding proof, we then conclude that  $P$ contains
a standard example $S_d$ with
\[
d\ge n-(5(6c+8)+12)=n-(30c+52).
\]
\end{proof}

\section{Closing Remarks}\label{sec:close}

As commented previously, our upper bound on $f(c)$ in
Theorem~\ref{thm:main-dim} shows that $f(c)=O(c^2)$. 
For a lower bound,
consider the poset $P$ associated with a finite projective plane of order $q$, 
as discussed in Example~1 in Section~\ref{sec:intro}.  Then let $m$ be
an integer which is large relative to $q$.  Form a poset $Q$ by
adding $2m$ new points to $P$. The new points form a standard example
$S_m$.  In $Q$, all minimal elements of $S_m$ are less than all maximal
elements of $P$, and all maximal elements of $S_m$ are greater than
all minimal elements of $P$.  Set $n=m+(q^2+q+1)$ and $c=q^{3/2}$.
Then $\dim(Q)\ge n-c$.  However,
$Q$ does not contain a standard example $S_d$ with $d\ge m+q^{3/2}+2$.
Since $q^2=c^{4/3}$, it follows that $f(c)=\Omega(c^{4/3})$.

With additional work, it is quite possible that the bounds on $f(c)$
may be tightened.  Of course, in the fractional dimension setting, 
it is quite possible that with further effort, the exact answer can 
be found, especially in the bipartite case.  

There are some more modest problems associated with the
details of our proofs.  One of them is the inequality
for bipartite posets: $\fdim(P)\le 2+\Idim^*(P)$.  Is
there a constant $q<2$ so that one always has $\fdim(P)\le
q+\Idim^*(P)$? We tend to believe that this holds when
$q=4/3$.  A second problem concerns the inequality of
Lemma~\ref{lem:two-standard}.  It is quite possible that
this inequality may be strengthened. 

A third problem is find the best possible bound
in Lemma~\ref{lem:TroMon}.  It is not too difficult to show that
the dimension of $P$ is at most $t+s$, when
$t$ is sufficiently large compared to $s$, so the real problem
is to find the maximum dimension when $t$ is bounded as a function
of $s$.

\section*{Acknowledgments}

The authors thank Tam\'as Sz\H{o}nyi, Ruidong Wang, and Bartosz Walczak for helpful ideas and conversations.

\end{document}